\documentclass[12pt]{amsart}

\usepackage{a4wide}
\usepackage{comment}
\usepackage[dvipsnames]{xcolor}
\usepackage{mathrsfs}
\usepackage{amsmath}
\usepackage{amsthm}
\usepackage{amsfonts}
\usepackage{amssymb}
\usepackage{enumitem}
\usepackage{graphicx}
\usepackage{palatino}
\usepackage{overpic}
\usepackage[T1]{fontenc}
\usepackage{hyperref}
\usepackage[noabbrev,capitalize,nameinlink]{cleveref}
\newtheorem{theorem}{Theorem}[section]
\newtheorem{proposition}[theorem]{Proposition}
\newtheorem{lemma}[theorem]{Lemma}

\theoremstyle{definition}
\newtheorem{definition}[theorem]{Definition}

\newtheorem{question}[theorem]{Question}

\theoremstyle{remark}
\newtheorem{remark}[theorem]{Remark}
\numberwithin{equation}{section}

\newcommand{\R}{\mathbb{R}}
\newcommand{\C}{\mathbb{C}}

\newcommand{\N}{\mathbb{N}}
\newcommand{\Z}{\mathbb{Z}}

\definecolor{darkgreen}{RGB}{0,162,0}
\definecolor{darkred}{RGB}{221,0,0}
\definecolor{darkblue}{RGB}{0,0,221}
\definecolor{gold}{RGB}{255,210,0}
\definecolor{red}{RGB}{242,43,29}
\hypersetup{colorlinks,linkcolor={blue},citecolor={darkgreen},urlcolor={darkgreen}}

\begin{document}

\title{Quasipositive fiber surfaces are not well-quasi-ordered}

\author{Joseph Breen}
\address{University of Alabama, Tuscaloosa, AL 35401}
\email{jjbreen@ua.edu} \urladdr{https://sites.google.com/view/joseph-breen}

\author{Michele Capovilla-Searle}
\address{University of Iowa, Iowa City, IA 52240}
\email{michele-capovilla-searle@uiowa.edu} 

\thanks{JB was partially supported by NSF Grant DMS-2038103 and an AMS-Simons Travel Grant. MCS was partially supported by an LLMJ
fellowship and NSF Grant DMS-2038103.}

\begin{abstract}
    By exhibiting an explicit infinite anti-chain, we show that the class of quasipositive fiber surfaces in $S^3$ is not well-quasi-ordered under the surface minor relation. This answers questions raised by Baader-Dehornoy-Liechti and Dehornoy-Lunel-de Mesmay in the negative.  
\end{abstract}

\maketitle

\tableofcontents

\section{Introduction}

\subsection{Background} Kruskal's tree theorem of 1960 \cite{kruskal1960tree}, conjectured by Vázsonyi in the 1940s and popularized by Erdös \cite{kruskal1972wqo}, states that finite trees are well-quasi-ordered with respect to the graph minor relation. In a series of celebrated papers written between 1980 \cite{robertson1983graphI} and 2004 \cite{robertson2004graphXX}, Robertson-Seymour generalized this from trees to all finite graphs, proving what was then known as Wagner's conjecture (and is now the eponymous theorem of Robertson and Seymour). 

A \textit{minor} is obtained from a graph by edge contraction and edge/vertex deletion, and \textit{well-quasi-orderedness} of the graph minor relation means that in every infinite sequence of finite graphs, there is always a pair of graphs such that one is a minor of the other. Well-quasi-orderedness of a minor relation is desirable because it implies that minor-closed properties --- properties preserved when passing to a minor, e.g. planarity of a graph --- are characterized by a finite set of forbidden minors.  

One can enrich this setting by considering spatial graphs $G \hookrightarrow S^3$ under a topological minor relation, and then further equipping the spatial graphs with a framing (c.f. \textit{ribbon graphs} \cite{ellismonaghan2013graphs}). Specifically, by associating embedded disks to each vertex and attaching bands along edges, one obtains a compact surface $\Sigma \supset G$ with boundary, what we henceforth refer to as a \textit{Seifert surface}. We call $G$, the embedded graph along which $\Sigma$ retracts, a \textit{spine}. The corresponding "topological framed graph minor" relation is then captured by:

\begin{definition}
Let $\Sigma, \Sigma'\subset S^3$ be Seifert surfaces. We say that $\Sigma$ \emph{is a surface minor of} $\Sigma'$, denoted $\Sigma \leq \Sigma'$, if $\Sigma$ is isotopic to a subsurface of $\Sigma'$ such that $\pi_1(\Sigma)$ injects into $\pi_1(\Sigma')$ under inclusion.\footnote{In the literature, such a subsurface is often called \textit{incompressible}, in reference to its incompressibility as a submanifold of $\Sigma'$.}
\end{definition}

One can immediately see that the set of Seifert surfaces is not well-quasi-ordered under the surface minor relation. Indeed, for $n\geq 0$, let $\Sigma_n$ be the ribbon surface of an $n$-framed unknot. Then $\Sigma_i$ is a surface minor of $\Sigma_j$ if and only if $i=j$, hence the sequence $\Sigma_{0}, \Sigma_1, \dots$ is an \textit{infinite anti-chain}, i.e. an infinite sequence in which no two elements are $\leq$-comparable. 

\subsection{Motivation} There has been interest in identifying classes of surfaces that are well-quasi-ordered under the surface minor relation. For example, Baader-Dehornoy-Liechti showed that fiber surfaces of links associated to isolated plane curve singularities are well-quasi-ordered \cite[Corollary 1]{baader2023minortheory}, as did Dehornoy-Lunel-Mesmay for the tree-like plumbings of Hopf bands known as Hopf arborescent surfaces \cite[Theorem 2]{dehornoy2024hopfminor}. The latter authors remark that, more generally, an optimistic conjecture is that the class of fibered surfaces is well-quasi-ordered \cite[p. 48:4]{dehornoy2024hopfminor}.

In this article we specialize to the class of fibered surfaces supporting the unique tight contact structure on $S^3$. By work of Hedden \cite{hedden2010notions}, these are precisely \textit{quasipositive fiber surfaces}, i.e. fiber surfaces for links arising as the closure of strongly quasipositive braids, a class which includes the algebraic links of \cite[Corollary 1]{baader2023minortheory} and the positive Hopf plumbings of \cite[Theorem 2]{dehornoy2024hopfminor}. 

Additional motivation in studying quasipositive surfaces can be found in the genus defect. In \cite{borodzik2019topological}, Borodzik-Feller showed that every link is topologically concordant to a strongly quasipositive link, hence the topological $4$-genus $g_4(L)$ is determined by its value on strongly quasipositive concordance class representatives. Quasipositive surfaces minimize the Seifert genus $g(L)$ of their boundary links, but they do not in generalize minimize $g_4(L)$. In fact, the \textit{genus defect} $\Delta g(L) = g(L) - g_4 (L)$ can be arbitrarily large. Importantly, $\Delta g$ is monotonic under the surface minor relation: $\Sigma \leq \Sigma'$ implies $\Delta g(\partial\Sigma)\leq \Delta g(\partial \Sigma')$. Moreover, for a fixed $N\in \N$, the property $\Delta g(\partial \Sigma) \leq N$ is minor-closed, hence can be described by a set of forbidden minors. Liechti \cite{liechti2016positive} gave a characterization of positive braid knots with maximal topological $4$-genus by three forbidden surface minors, and subsequently \cite{liechti2020genus} used this theory to show that the $\Delta g$ is bounded below by an affine function of the braid index.

\subsection{Main result}

The following is suggested by the questions and results of \cite{baader2023minortheory,dehornoy2024hopfminor}, in particular the discussion in \cite[p. 48:4]{dehornoy2024hopfminor}.

\begin{question}\label{question:1}
Let $\Sigma_1, \Sigma_2, \dots \subset S^3$ be an infinite sequence of quasipositive fiber surfaces. Must there exist a pair $\Sigma_i, \Sigma_j$ with $i\neq j$ such that $\Sigma_i \leq \Sigma_j$? That is, are quasipositive fiber surfaces well-quasi-ordered under the surface minor relation?
\end{question}

Our main result answers this question in the negative, dispelling the optimistic conjecture alluded to by \cite{dehornoy2024hopfminor}.

\begin{theorem}\label{thm:main}
For $n\geq 1$, let $\Sigma_{2n+1}$ denote the Seifert surface in \cref{fig:torus_family}. Then 
\begin{enumerate}
    \item Each $\Sigma_{2n+1}$ is a quasipositive fiber surface. 
    \item For all $m\neq n$, the surface $\Sigma_{2m+1}$ is not a surface minor of $\Sigma_{2n+1}$.
\end{enumerate}
In particular, \cref{question:1} admits a negative answer. 
\end{theorem}

\begin{figure}[ht]
	\centering
	\begin{overpic}[scale=.44]{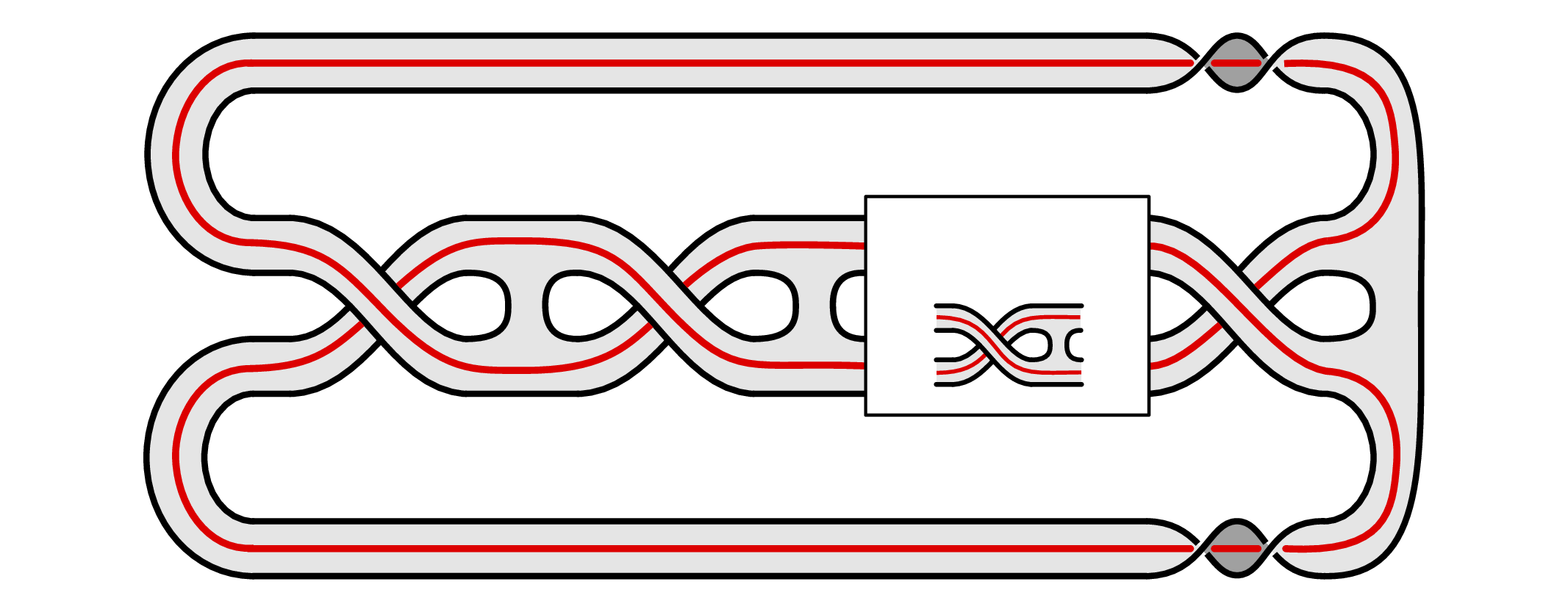}
        \put(1,34){\large $\Sigma_{2n+1}$}
        \put(17,9){\textcolor{darkred}{$T(2,2n+1)$}}
        \put(56.25,22){\footnotesize $2n-2$ copies of:}
	\end{overpic}
	\caption{The family of quasipositive fiber surfaces in $S^3$ in \cref{thm:main}. The red curve embedded in $\Sigma_{2n+1}$ is a $(2,2n+1)$ torus knot.}
	\label{fig:torus_family}
\end{figure}

\subsection{Organization} In \cref{sec:background}, we provide some background on contact structures and Legendrian knots. In \cref{sec:qp_fib} we use contact topology to prove (1) of \cref{thm:main}; namely, the fact that the surfaces $\Sigma_{2n+1}$ are quasipositive and fibered. Finally, we perform Seifert form calculations in \cref{sec:non_wqo} to verify (2) of \cref{thm:main}.

\subsection{Acknowledgments} The authors would like to thank Keiko Kawamuro and Sebastian Baader for commenting on a preliminary draft. 

\section{Background}\label{sec:background}

Here we cover some brief background as is needed in subsequent sections, primarily on contact topology. Standard contact references include the book of Geiges \cite{geiges2008introduction} and the lecture notes of Etnyre \cite{etnyre2005surveyknots,etnyre2006lectures}.

\subsection{Knot and braids} 

A \textit{framing} of a knot $K\subset S^3$ is a trivialization of its normal bundle. Such a framing is determined by a nonzero section of the normal bundle, hence by a nonvanishing transverse vector field defined locally along $K$ or equivalently a choice of parallel copy of $K$. One can associate an integer to a choice of framing as follows. First, choose an oriented Seifert surface $F$ for $K$, and orient $K$ as $\partial F$. Choose a parallel pushoff $K'$ of $K$ which represents the desired framing, and orient it so that its direction agrees with that of $K$. The \textit{framing integer} associated to the framing is then by definition the algebraic intersection number $K' \cdot \Sigma$. Note that the \textit{Seifert framing}, the framing determined by a parallel pushoff of $K$ in the direction of any Seifert surface, is the $0$-framing. 

More generally, the \textit{linking number} between two oriented knots $K_1, K_2\subset S^3$ is defined by $\ell k(K_1, K_2):=K_2 \cdot \Sigma_1$, where $\Sigma_1$ is a Seifert surface for $K_1$. Diagrammatically, we have 
\[
\ell k(K_1, K_2) = \frac{(\textrm{\# positive $K_1,K_2$ crossings}) - (\textrm{\# negative $K_1,K_2$ crossings})}{2}
\]
where a positive (resp. negative) crossing has an overpass going to the right (resp. left). 

A link $L\subset S^3$ is \textit{fibered} if $S^3 \setminus L$ fibers over $S^1$, such that each fiber compactifies to a Seifert surface of $L$. A link is \textit{strongly quasipositive} if it is isotopic the closure of a \textit{strongly quasipositive braid} --- one which is a positive word in the conjugates $a_{ij}:=(\sigma_{j-1}\sigma_{j-2}\cdots \sigma_{i+1})\sigma_{i}(\sigma_{j-1}\sigma_{j-2}\cdots \sigma_{i+1})^{-1}$ of the standard Artin generators, a class which was first introduced and studied by Rudolph \cite{rudolph1990congruence,rudolph1992characterization,rudolph1993quasipositivity}.  Given such a braid word, one can identify a canonical Seifert surface for its closure by starting with disks corresponding to the number of strands in the braid and attaching negatively-twisted bands between strands $i$ and $j$ associated to each letter $a_{ij}$; see \cref{fig:SQP}. We call any surface arising this way (the canonical Seifert surface associated to a strongly quasipositive braid closure) a \textit{quasipositive} surface.

\begin{figure}[ht]
	\centering
	\begin{overpic}[scale=.34]{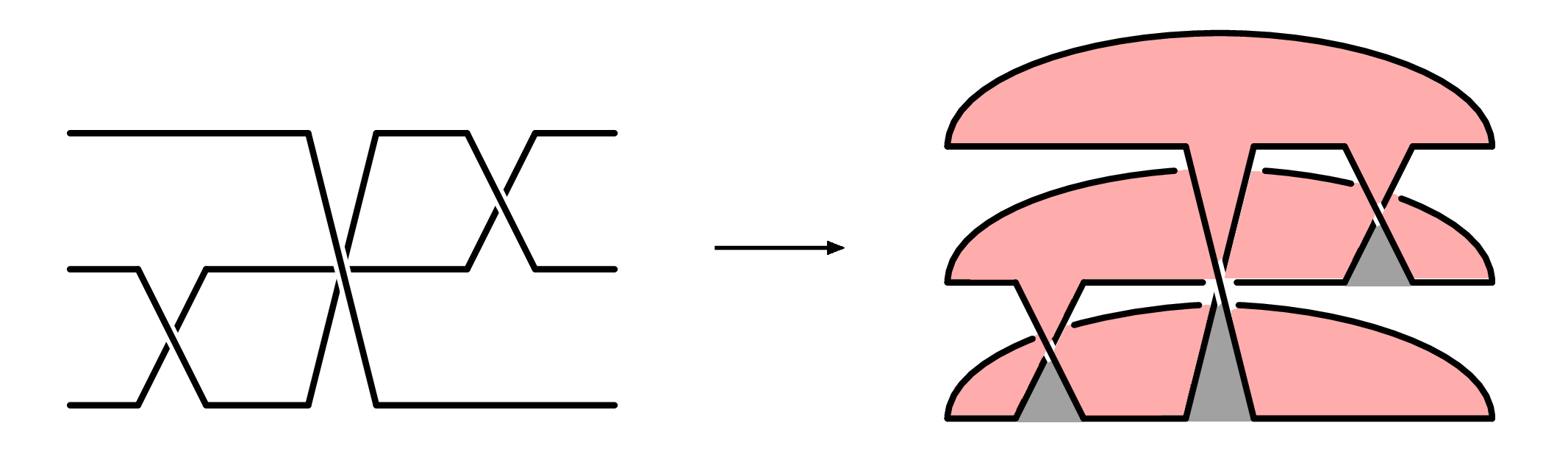}
        
	\end{overpic}
	\caption{The strongly quasipositive braid $a_{12}a_{13}a_{23}$ and the quasipositive Seifert surface of its closure.}
	\label{fig:SQP}
\end{figure}

\subsection{Contact structures}

A \textit{(co-oriented and positive) contact structure} on an oriented $3$-manifold $M$ is a plane distribution $\xi$ for which there exists a globally defined $1$-form $\alpha$ satisfying $\xi = \mathrm{ker}(\alpha)$ and $\alpha \wedge d\alpha > 0$. Two contact manifolds $(M,\xi)$ and $(N,\zeta)$ are \textit{contactomorphic} if there exists a diffeomorphism $\varphi : M \to N$ such that $\varphi_*(\xi) = \zeta$.

There are in general many contact structures on any chosen $3$-manifold. Focusing on $\mathbb R^3$ with coordinates $(x,y,z)$, the \textit{standard} contact structure is the kernel of $\alpha_{\mathrm{st}} := dz - y\, dx$. This compactifies to a contact structure on $S^3$, also called the \textit{standard} structure, which is contactomorphic to the plane field of complex tangencies $T_pS^3 \cap i(T_pS^3)$ to the unit sphere $S^3 \subset \C^2$. A structure non-contactomorphic to $\xi_{\mathrm{st}}$ on $\mathbb R^3$ is given in cylindrical coordinates by $\xi_{\mathrm{OT}}:=\ker(\cos(r)\, dz + r\sin(r)\, d\theta)$. The feature that distinguishes $\xi_{\mathrm{OT}}$ from $\xi_{\mathrm{st}}$ is the presence of an \textit{overtwisted disk}, an embedded disk which is tangent to the contact structure along its boundary. In $\xi_{\mathrm{OT}}$ this is given by $D = \{z=0,r\leq \pi\}$. Any contact structure which lacks an overtwisted disk (such as $\xi_{\mathrm{st}}$) is called \textit{tight}.

An \textit{abstract open book} $(\Sigma, \phi)$ is a compact oriented surface $\Sigma$ with non-empty boundary, called the \textit{page}, and a diffeomorphism $\phi: \Sigma \to \Sigma$ fixing $\partial \Sigma$ pointwise, called the \textit{monodromy}. Given an abstract open book $(\Sigma,\phi)$ we define a $3$-manifold $M_{(\Sigma,\phi)}$ as follows. Consider the mapping torus $M_\phi := \Sigma \times [0,1]_t / \sim$ where $(x,1) \sim (\phi(x),0)$. Since the diffeomorphism fixes the $\partial \Sigma$ pointwise, $\partial M_{\phi}$ is a disjoint union of tori. We obtain $M_{(\Sigma,\phi)}$ by gluing in solid tori $D^2 \times S^1$ along $\partial M_{\phi}$ so that for each point $p\in \partial \Sigma$, the circle $\{p\}\times S^1 \subset \partial M_{\phi}$ bounds a meridian disk of $D^2 \times S^1$. 

In fact, Thurston-Winkelnkemper \cite{thurston1975existence} showed that $M_{(\Sigma, \phi)}$ inherits a well-defined contact structure $\xi_{(\Sigma,\phi)}$. Briefly, since $\Sigma$ is a compact orientable surface with boundary, we may endow it with an exact symplectic form $d\lambda$; the mapping torus $M_{\phi}$ then inherits a contact structure given by the contactization $\ker(dt + \lambda)$. One can then show that it is possible to fill in the solid tori with a standard contact model to define a contact structure on $M_{(\Sigma, \phi)}$. If $(M, \xi)$ is contactomorphic to $(M_{(\Sigma, \phi)}, \xi_{(\Sigma,\phi)})$, we say that $(M, \xi)$ is \textit{supported} by the abstract open book $(\Sigma, \phi)$. 

By construction, the link $B:= \bigsqcup\,\,(\{0\}\times S^1)$ in $M_{(\Sigma, \phi)}$ is a fibered link, and conversely, any fibered link $B$ in a $3$-manifold yields a diffeomorphism with $M_{(\Sigma,\phi)}$ where $\Sigma$ is diffeomorphic to a Seifert surface for $B$ and $\phi$ is the monodromy of the fibration. Thus, a choice of fibered link $B\subset M$ is also called an \textit{open book decomposition} of $M$ and $B$ is called the \textit{binding}. Consequently, a choice of fibered link in a $3$-manifold induces a contact structure by passing through the corresponding abstract open book $(\Sigma, \phi)$ and Thurston-Winkelnkemper's construction. 

As referenced in the introduction, work of Hedden \cite{hedden2010notions} shows that a fibered link is strongly quasipositive if and only if its corresponding open book decomposition supports the standard contact structure on $S^3$.

\subsection{Legendrian knots}

A \textit{Legendrian} knot $\Lambda\subset (M, \xi)$ in a contact manifold is a smooth knot that is always tangent to $\xi$; that is, $T_p\Lambda \in \xi_p$ for all $p\in \Lambda$. One typically studies Legendrian knots up to \textit{Legendrian isotopy}, i.e. isotopy through Legendrian knots.

There are two geometrically meaningful projections of a Legendrian knot in $(\mathbb R^3, \xi_{\mathrm{st}})$: the \textit{front projection} to the $(x,z)$-plane, and the \textit{Lagrangian projection} to the $(x,y)$-plane. Front projections have semi-cubical cusps in place of vertical tangencies, and the $y$-coordinate can be recovered from $y=\frac{dz}{dx}$. In particular, overcrossings always have more negative slope that undercrossings.

Given a front projection of a Legendrian knot, one can obtain (up to Legendrian isotopy) its Lagrangian projection (up to planar isotopy) by smoothing the cusps in the front projection according to \cref{fig:ng-resolution}; see \cite{ng2003resolution}.

\begin{figure}[ht]
	\centering
	\begin{overpic}[scale=.27]{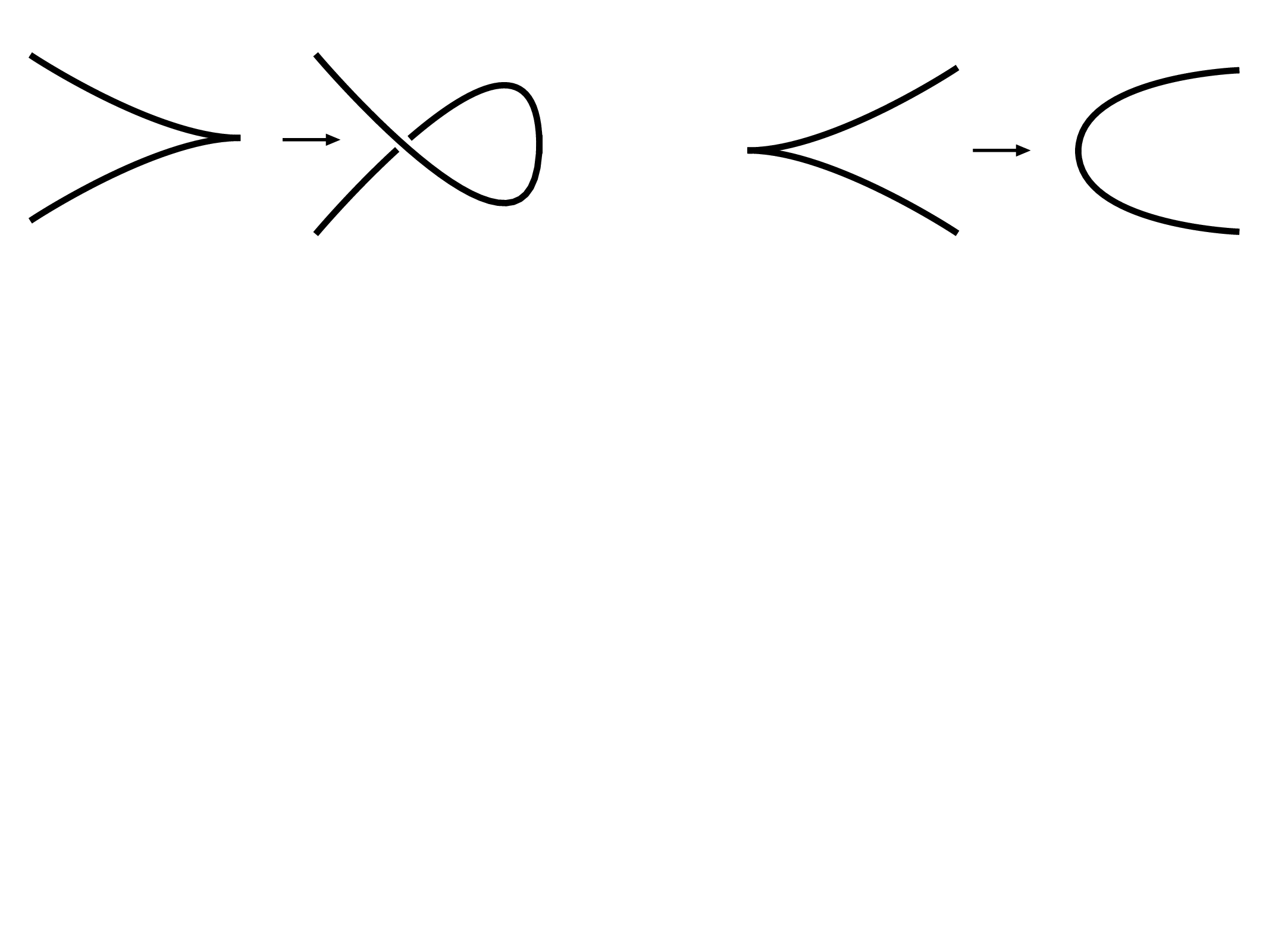}
        
	\end{overpic}
    \vskip-5cm
	\caption{The Ng resolution of a front projection.}
	\label{fig:ng-resolution}
\end{figure}

A Legendrian knot $\Lambda$ is equipped with a framing coming from the contact structure; specifically, given by a parallel pushoff of $\Lambda$ in a direction transverse to $\xi$. The associated framing integer is known as the \textit{Thurston-Bennequin number}, denoted $\mathrm{tb}(\Lambda)$, and is one of the so-called classical invariants of a Legendrian knot. We will appeal to two facts about the contact framing and the Thurston-Bennequin number: 
\begin{enumerate}
    \item In the front projection, $\mathrm{tb}(\Lambda)$ can be computed combinatorially as 
    \begin{align*}
        \mathrm{tb}(\Lambda) &= \textrm{writhe} - \frac{1}{2}(\# \textrm{ of cusps}) \\
        &= (\# \textrm{ of pos. crossings}) - (\# \textrm{ of neg. crossings}) -\frac{1}{2}(\# \textrm{ of cusps}).
    \end{align*}
    For example, the Legendrian unknot $U$ pictured in the top row of \cref{fig:torus_knot} has $\mathrm{tb}(U) = -1$. The Legendrian $(2,2n+1)$-torus knot pictured in the second row has $2n+1$ positive crossings and $4$ cusps, hence $\mathrm{tb}(\Lambda) = 2n-1$. These knots are referred to as the \textit{max-tb} unknot and $(2,2n+1)$-torus knot, respectively, as they maximize the Thurston-Bennequin number in their topological knot type in $(S^3, \xi_{\mathrm{st}})$.

    \item In the Lagrangian projection, the contact framing of a Legendrian knot coincides with the blackboard framing; indeed the vector field $\partial_z$ (which is projected out in the Lagrangian projection) is transverse to $\xi_{\mathrm{st}} = \ker(dz-y\,dx).$
\end{enumerate}

\begin{figure}[ht]
	\centering
	\begin{overpic}[scale=.4]{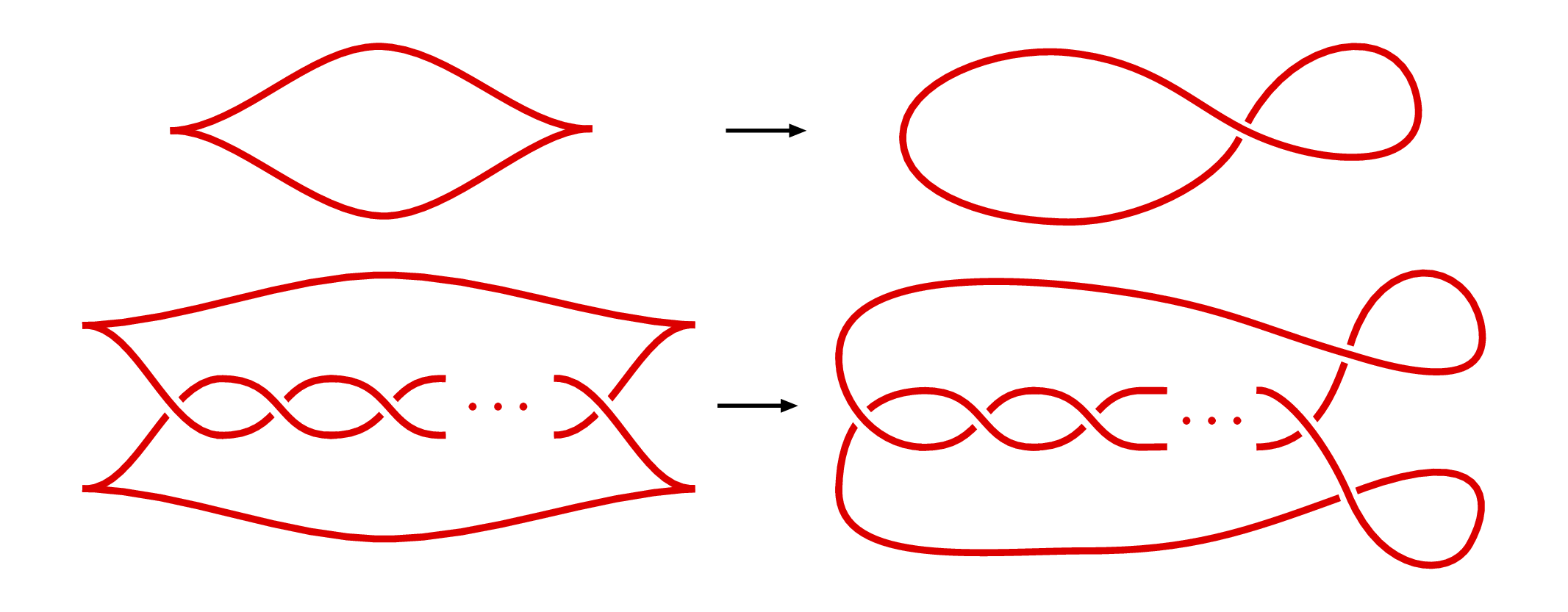}
       
	\end{overpic}
    
	\caption{In the top row, a Legendrian unknot in the front (left) and Lagrangian (right) projection. In the bottom row, a Legendrian $(2,2n+1)$-torus knot in the front (left) and Lagrangian (right) projection.}
	\label{fig:torus_knot}
\end{figure}

\section{Proof of quasipositive fiberedness}\label{sec:qp_fib}

In this section we prove that $\Sigma_{2n+1}$ is a quasipositive fiber surface, establishing (1) of \cref{thm:main}. In short, the surface $\Sigma_{2n+1}$ is the result of applying an algorithm of Avdek \cite{avdek2013surgery} --- see also the procedure in \cite[\S 3]{islambouli2024multisections} --- to the max-tb Legendrian $(2,2n+1)$ torus knot. Avdek's algorithm produces an open book decomposition supporting the tight contact structure on $S^3$ for which the input knot is embedded on a page, such that the page-framing of the knot coincides with the contact framing; the resulting page is our quasipositive fiber surface $\Sigma_{2n+1}$. We provide additional details to give a mostly self-contained exposition.  

\begin{remark}
Quasipositive fiberedness holds for the corresponding even-indexed surfaces $\Sigma_{2n}$ as well. We ultimately restrict to odd indices for the convenience that $\Lambda_{2n+1}$ is a knot, rather than a link.    
\end{remark}

\begin{figure}[ht]
	\centering
	\begin{overpic}[scale=.44]{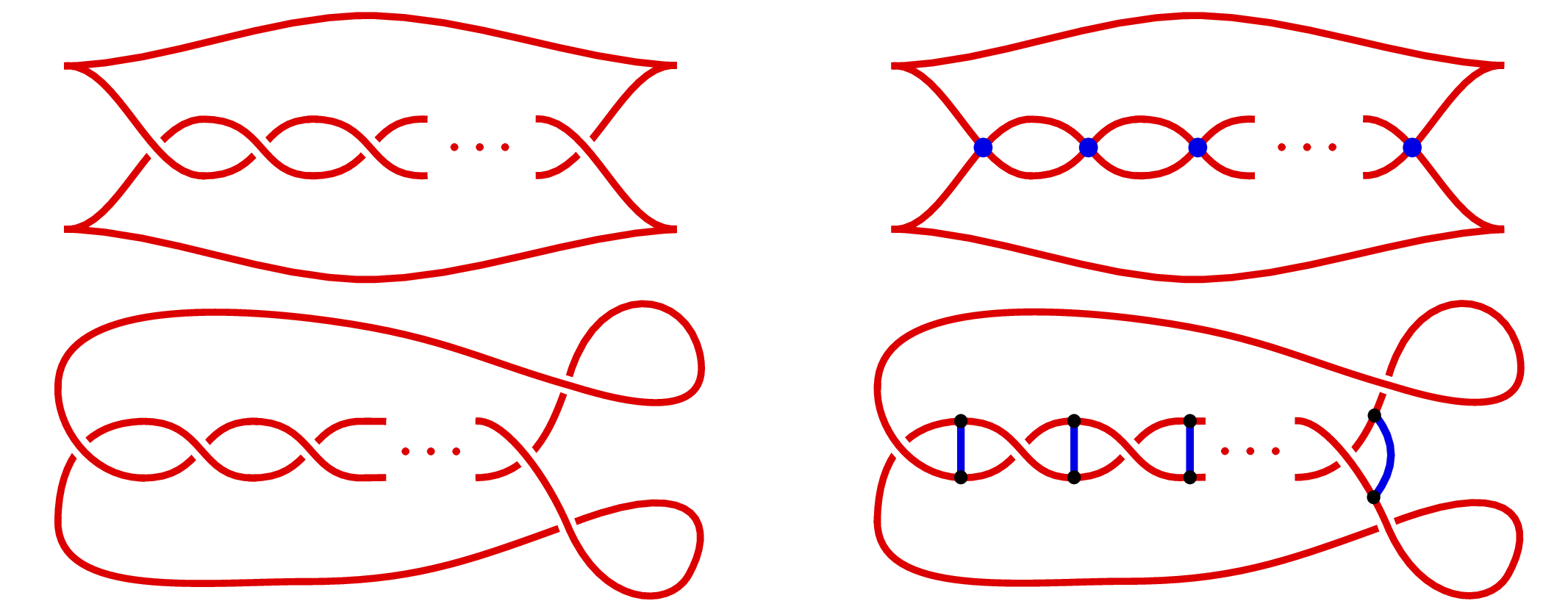}
        \put(0,20.5){\small \textcolor{darkred}{$\Lambda_{2n+1}$}}

        \put(52.5,20.5){\small \textcolor{black}{$G_{2n+1}$}}
	\end{overpic}
	\caption{}
	\label{fig:torus_family_graph}
\end{figure}

Let $\Lambda_{2n+1}$ be the max-tb Legendrian $(2,2n+1)$ torus knot, whose front projection (with $2n+1$ crossings) in $(\R^3, dz - y\, dx)$ is given by the top left of \cref{fig:torus_family_graph}. Let $G_{2n+1}$ be the Legendrian graph obtained by attaching edges parallel to the $y$-axis at each crossing; these edges, invisible in the front projection, are indicated by the blue dots in the top right of the figure. The graph, up to Legendrian graph isotopy, may be visualized in the Lagrangian projection by first performing the Ng resolution of the front projection, indicated in the bottom left of \cref{fig:torus_family_graph}, and then attaching the indicated edges as in the bottom right of the figure.

Let $\Sigma_{2n+1}$ be the \textit{Legendrian ribbon} of $G_{2n+1}$, i.e. the surface, unique-up-to-isotopy, which retracts onto $G_{2n+1}$, is tangent to $\xi_{\mathrm{st}}$ along $G_{2n+1}$, and is transverse to $\xi_{\mathrm{st}}$ away from $G_{2n+1}$; see \cite[\S 2]{baader2009graphs}, or \cite{Giroux2002GeometrieDC,etnyre2006lectures}. One may identify $\Sigma_{2n+1}$ in the Lagrangian projection by taking the blackboard push-off of the graph $G_{2n+1}$ in the bottom right of \cref{fig:torus_family_graph}, as the contact framing agrees with the blackboard framing in this projection. The resulting surface is indeed isotopic to $\Sigma_{2n+1}$ as shown in \cref{fig:torus_family}. Since $\Sigma_{2n+1}$ is the Legendrian ribbon of a Legendrian graph, \cite[Theorem 2.2]{baader2009graphs} implies that $\Sigma_{2n+1}$ is a quasipositive surface. 

It remains to show that $\Sigma_{2n+1}$ is a fiber surface. To that end, we will show that $G_{2n+1}$ is the $1$-skeleton of a suitable \textit{contact cell decomposition} of $(S^3, \xi_{\mathrm{st}})$ and appeal to the following theorem, originally due to Giroux \cite{Giroux2002GeometrieDC}, with a proof explained by Etnyre \cite{etnyre2006lectures}, and stated by Avdek as \cite[Theorem 2.4]{avdek2013surgery}. (Applying this theorem provides a separate and independent proof that $\Sigma_{2n+1}$ is quasipositive, thanks to Hedden \cite{hedden2010notions}.)

\begin{theorem}\cite{Giroux2002GeometrieDC}\label{thm:giroux}
Suppose that a closed co-oriented contact manifold $(M,\xi)$ has a cell decomposition with the following properties. 
\begin{enumerate}
    \item The $1$-skeleton $G$ is a Legendrian graph.
    \item Each $2$-cell $D$ is a convex (in the sense of \cite{giroux1991convexite}) disk with piecewise-Legendrian boundary such that, after performing the unique Legendrian smoothing of $\partial D$, we have $\mathrm{tb}(\partial D) = -1$. 
    \item If $\Sigma$ denotes the Legendrian ribbon of $G$, then for every $2$-cell $D$, we have $|D\cap \partial \Sigma| = 2$.
    \item Every $3$-cell is tight. 
\end{enumerate}
Then $\Sigma$ is the page of a supporting open book decomposition of $(M, \xi)$. 
\end{theorem}

\begin{remark}
The idea behind \cref{thm:giroux} is that a standard tubular neighborhood of the Legendrian $1$-skeleton $G$ forms a handlebody $H_0$ contactomorphic to a contactization $([-1,1]_t\times \Sigma, \ker(dt + \lambda)$ of the (conformally) symplectic surface $(\Sigma, d\lambda)$; properties (2-4) ensure that the complement $M- H_0$ is also a handlebody contactomorphic to $[-1,1]\times \Sigma$. The two handlebodies glue together to form a contact Heegaard splitting that arises from a supporting open book decomposition. This perspective is further explained in, for instance, \cite{honda2009hkm_contact_invariant_sutured,ozbagci2011contact,islambouli2024multisections}.    
\end{remark}

\begin{proof}[Proof of (1) of \cref{thm:main}.]
We show that the Legendrian graph $G_{2n+1}$ is a $1$-skeleton for a cell decomposition of $(S^3, \xi_{\mathrm{st}})$ satisfying \cref{thm:giroux}. As $(S^3, \xi_{\mathrm{st}})$ is tight, (4) is automatic, so it remains to identify the $2$-cells and verify (2) and (3). We identify $2$-cells $D_1, D_2, \dots, D_{2n+1}, D_{2n+2}$ as follows. The cell $D_1$ (resp. $D_{2n+2}$) corresponds to the natural lift of the large lower (resp. upper) planar disk cut out by the front projection of $G_{2n+1}$ in the top right of \cref{fig:torus_family_graph}; in \cite[\S 4]{avdek2013surgery} these planar disks cut out by Legendrian graph projections are called \textit{elementary disks}. Up to a $C^0$-small perturbation relative to the smoothed boundaries, we may assume that each disk is convex \cite{honda2000classification}. The disk $D_{2n+2}$ is shaded in gray in the top right of \cref{fig:torus_family_disks} and is shaded in the Lagrangian projection in the bottom right. A slight inward-pushoff of the natural smoothing of $\partial D_{2n+2}$ is also drawn as a dashed black curve, and is evidently a $\mathrm{tb}=-1$ unknot. The disks $D_2, \dots, D_{2n+1}$ correspond to the smaller inner elementary disks cut out by the twist region and the bridges in the front projection; $D_2$ and a pushoff of its $\mathrm{tb}=-1$ boundary are drawn on the left side of \cref{fig:torus_family_disks}. 

\begin{figure}[ht]
	\centering
	\begin{overpic}[scale=.44]{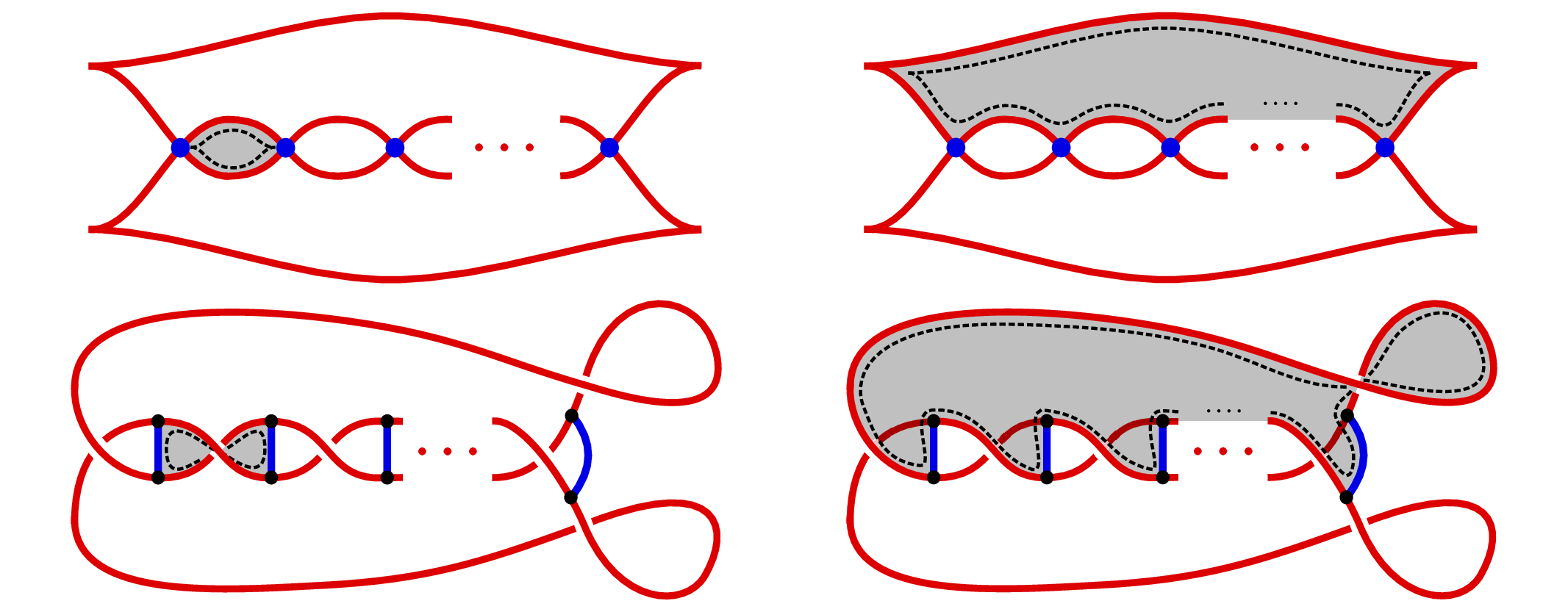}

        \put(63.5,15){\small $D_{2n+2}$}
        \put(70.5,34){\small $D_{2n+2}$}
	\end{overpic}
	\caption{}
	\label{fig:torus_family_disks}
\end{figure}

The argument then proceeds as in \cite[\S 4]{avdek2013surgery}. First, observe that the union of the $1$-skeleton and $2$-skeleton, $G_{2n+1} \cup \bigcup D_i$, is contractible. Its complement is then homeomorphic to a $3$-ball, so to complete a cell decomposition of $S^3$ we only need to attach a single $3$-cell. As already observed, tightness of the $3$-cell is automatic from tightness of the ambient contact structure.

It remains to check the intersection condition (3) of \cref{thm:giroux}. The verification for the disks $D_2, \dots, D_{2n+1}$ is described in \cref{fig:intersection_1}. In particular, beginning with the local structure of $G_{2n+1}$ near one such disk in the Lagrangian projection (far left), we identify $\partial \Sigma_{2n+1}$ in black via a small blackboard pushoff (middle). Up to a generic perturbation of the disk, $\partial \Sigma_{2n+1}$ punctures the disk twice transversely, as indicated. The far right figure is a further smooth isotopy of $\partial \Sigma_{2n+1}$ included for topological clarification. 

\begin{figure}[ht]
	\centering
	\begin{overpic}[scale=.42]{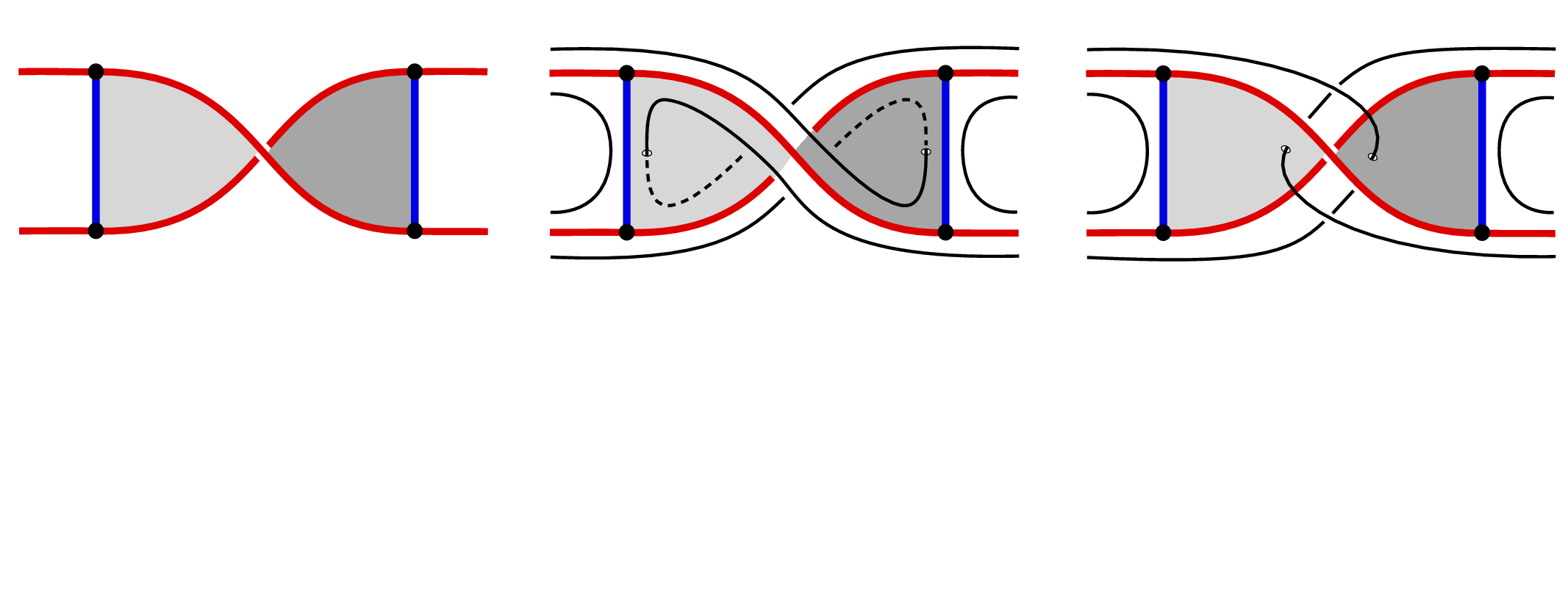}
      
	\end{overpic}
    \vskip-3cm
	\caption{}
	\label{fig:intersection_1}
\end{figure}

The verification for the disk $D_{2n+2}$ is contained in \cref{fig:intersection_2}. The different components of the blackboard pushoff of $G_{2n+1}$ are drawn in the top portion of the figure in green, blue, and black for visual distinction. The two intersections with $D_{2n+2}$ (one in black, one in green) are indicated near the left and right sides of $D_{2n+2}$. The other components of $\partial \Sigma_{2n+1}$ --- for example, the component in blue labeled $K$ --- may be taken to sit below $D_{2n+2}$ without any necessary intersections. The lower part of the figure performs a topological isotopy on such a component for clarification. By symmetry, the verification of the intersection condition for the lower disk $D_1$ is identical to that of $D_{2n+2}$.
\end{proof}

\begin{figure}[ht]
	\centering
	\begin{overpic}[scale=.42]{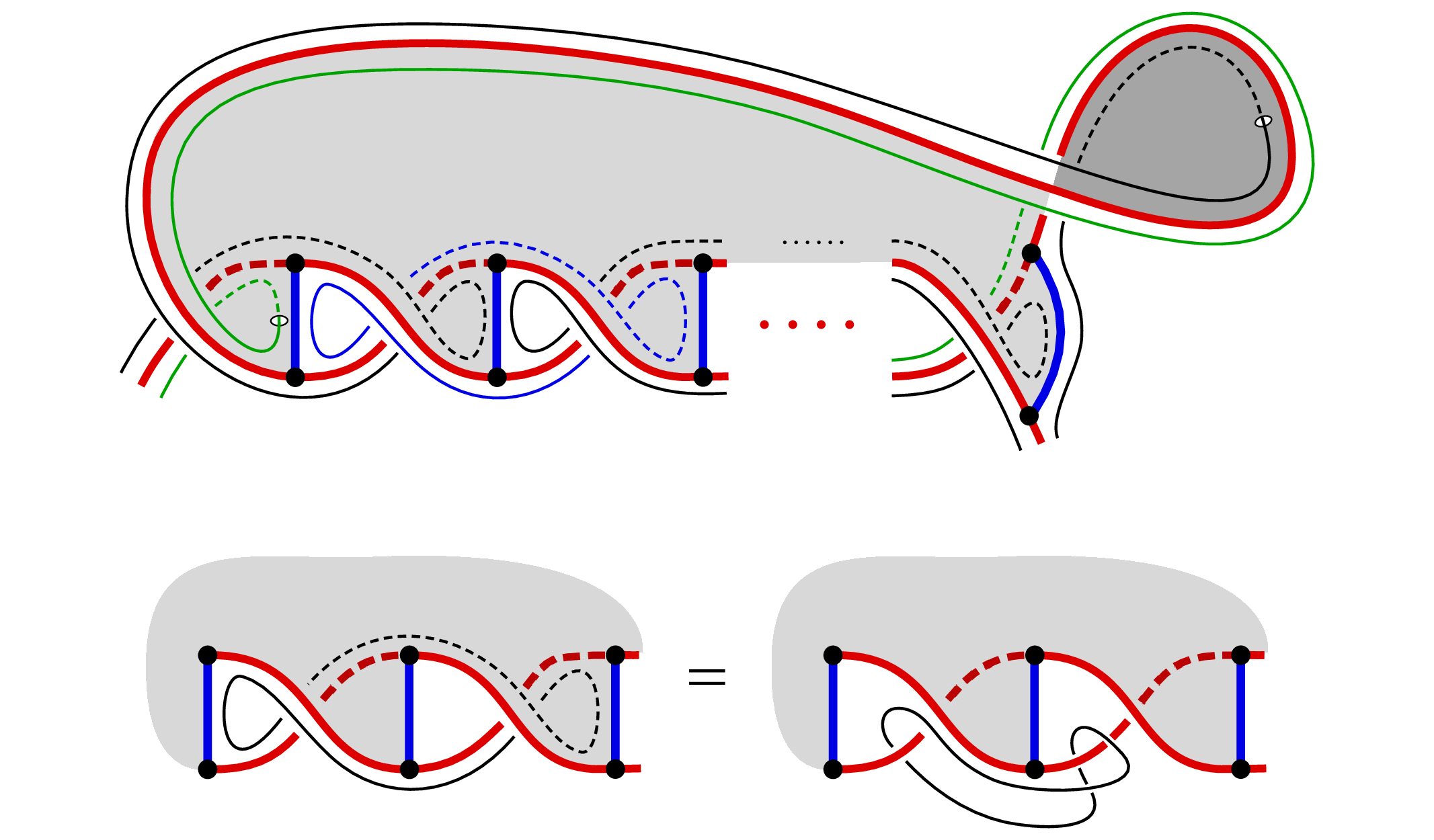}
      \put(33.5,28){\textcolor{darkblue}{$K$}}

      \put(33.5,46){$D_{2n+2}$}
	\end{overpic}
	\caption{}
	\label{fig:intersection_2}
\end{figure}

\section{Proof of non-well-quasi-orderedness}\label{sec:non_wqo}

Associated to a Seifert surface $\Sigma\subset S^3$ is its \textit{Seifert form}
\[
V_{\Sigma} : H_1(\Sigma; \mathbb Z) \times H_1(\Sigma;\mathbb Z) \to \mathbb Z,
\]
defined for simple closed oriented curves $a,b \subset \Sigma$ by $V_{\Sigma}([a],[b]) := \ell k(a,b^+)$, where $b^+$ is a push-off in the positive normal direction of $\Sigma$. We will denote an unconventional symmetrization (differing by an inconsequential factor of $\frac{1}{2}$ from the usual one) of the Seifert form by
\[
\tilde V_{\Sigma}:= V_{\Sigma}+ V_{\Sigma}^T
\]
where $V_{\Sigma}^T([a],[b]) := \ell k(b,a^+)$.

\begin{proposition}\label{unique-eigenvalue}
For each $n\in \mathbb N$, there is a basis of $H_1(\Sigma_{2n+1}; \Z)\cong \Z^{2n+2}$ consisting of simple closed curves for which the symmetrized Seifert form $\tilde{V}_{2n+1}:= \tilde{V}_{\Sigma_{2n+1}}$ is represented by the $(2n+2) \times (2n+2)$ matrix
\[
\tilde{A}_{2n+1} := 
\begin{pmatrix}
-2 & 1 & 1 & \cdots & 1 & 2n+1 \\
1 & -2 & 0 & \cdots & 0 & -1 \\
1 & 0 & -2 & \cdots & 0 & -1 \\
\vdots & \vdots & \vdots & \ddots & \vdots & \vdots\\
1 & 0 & 0 & \cdots & -2 & -1 \\
2n+1 & -1 & -1 & \cdots & -1 & -2
\end{pmatrix}.
\]
\end{proposition}

\begin{proof}
Note as in \cref{fig:torus_family} the surface $\Sigma_{2n+1}$ is comprised of two annuli $\mathcal{A}_1, \mathcal{A}_2$ with a subsequent attachment of $2n$ 1-handles. An Euler characteristic computation shows that
\[
\chi(\Sigma_{2n+1}) = \chi(\mathcal{A}_1) + \chi (\mathcal{A}_2) - 2n = -2n.
\]
As a result, $H_1(\Sigma_{2n+1}; \Z) \cong \mathbb Z^{2n+2}$. A basis $\{\gamma_1 , \dots , \gamma_{2n+2}\}$ of oriented curves is depicted in \cref{fig:torus_family_curves}.

\begin{figure}[ht]
	\centering
	\begin{overpic}[scale=.44]{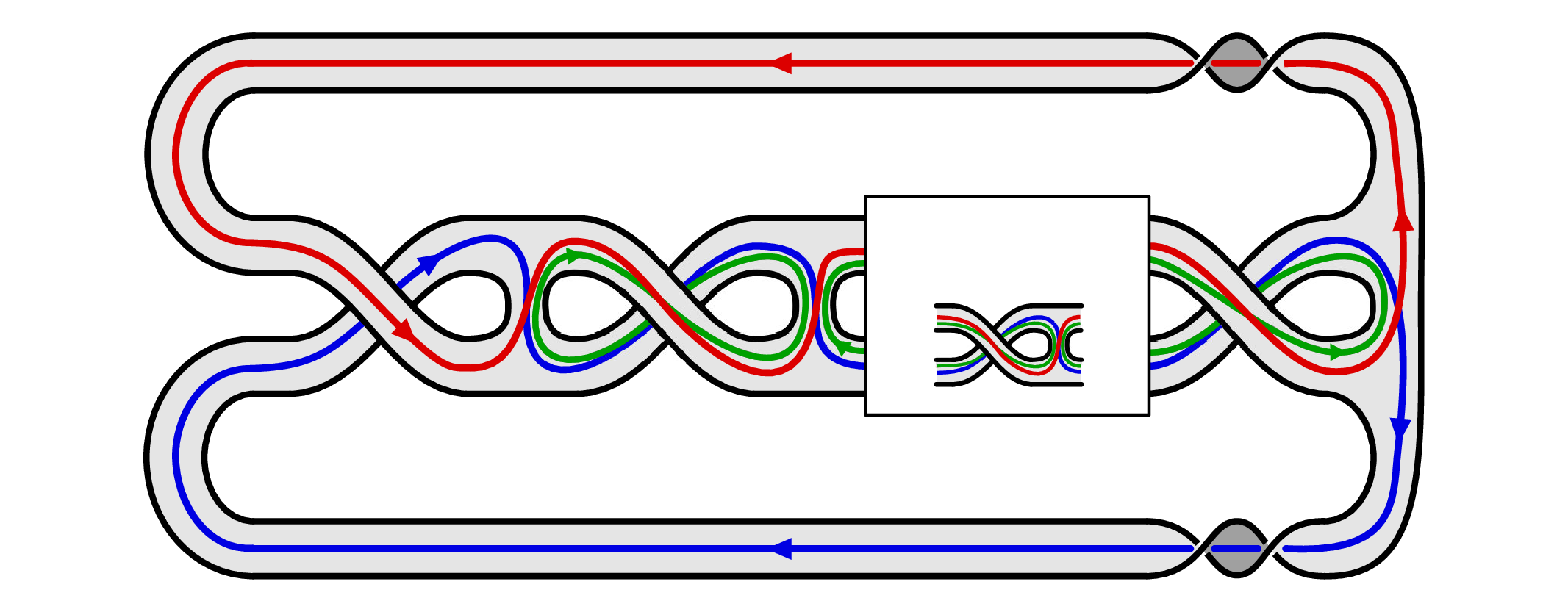}
        \put(1,34){\large $\Sigma_{2n+1}$}
        \put(16,9){\large \textcolor{darkblue}{$\gamma_1$}}
        \put(16,28.5){\large \textcolor{darkred}{$\gamma_{2n+2}$}}
        \put(41.5,13.5){\large \textcolor{darkgreen}{$\gamma_2$}}
        \put(75.5,12.5){\large \textcolor{darkgreen}{$\gamma_{2n+1}$}}
        \put(56.25,22){\footnotesize $2n-2$ copies of:}
	\end{overpic}
	\caption{}
	\label{fig:torus_family_curves}
\end{figure}

We first compute the matrix representative for $V_{2n+1} := V_{\Sigma_{2n+1}}$. Here we orient $\Sigma_{2n+1}$ so that its co-orientation points out of the page in the lightly shaded region, so that positive push-offs are accordingly out of the page. For each $i$, we then have $V_{2n+1}(\gamma_i, \gamma_i) = -1$: for $i=1,2n+2$ (i.e. blue and red), this is a consequence of the negative twist in the surface near the top right (resp. bottom left) region, while for $i=2, \dots, 2n+1$ (i.e. green) this is due to the negative crossing in the projection of each curve.\footnote{Assuming the basis elements are Legendrian, we have $V_{2n+1}(\gamma_i, \gamma_i) = \mathrm{tb}(\gamma_i)$ by definition of the Thurston-Bennequin number and the fact that $\Sigma$ is a Legendrian ribbon neighborhood of a Legendrian graph. Each knot is a max-tb unknot, giving a contact topological interpretation of $V_{2n+1}(\gamma_i, \gamma_i) = -1$.} By inspection, we see $V_{2n+1}(\gamma_1, \gamma_{2n+2}) = 0$, whereas $V_{2n+1}(\gamma_{2n+2}, \gamma_{1}) = 2n+1$, given the $2(2n+1)$ positive crossings of $\gamma_{2n+2}$ and $\gamma_1^+$. Inspection also gives $V_{2n+1}(\gamma_1, \gamma_{j}) = V_{2n+1}(\gamma_i,\gamma_{2n+2}) = V_{2n+1}(\gamma_i,\gamma_{j}) = 0$ for all $1\leq i,j\leq 2n+1$. Finally, by counting the signs of intersections of green with blue and red, we have $V_{2n+1}(\gamma_{2n+2},\gamma_j) = -1$ and $V_{2n+1}(\gamma_i,\gamma_1) = 1$ for $1\leq i,j\leq 2n+1$. This gives a matrix representative of $V_{2n+1}$ as 
\[
A_{2n+1} = (V_{2n+1}(\gamma_i,\gamma_j)) = 
\begin{pmatrix}
-1 & 0 & 0 & \cdots & 0 & 0 \\
1 & -1 & 0 & \cdots & 0 & 0 \\
1 & 0 & -1 & \cdots & 0 & 0 \\
\vdots & \vdots & \vdots & \ddots & \vdots & \vdots\\
1 & 0 & 0 & \cdots & -1 & 0 \\
2n+1 & -1 & -1 & \cdots & -1 & -1
\end{pmatrix}.
\]
The symmetrized matrix $\tilde{A}_{2n+1} = A_{2n+1} + A^{T}_{2n+1}$ then follows immediately. 
\end{proof}

We will use the following classical theorem to estimate the signature of the symmetrized Seifert form. 

\begin{theorem}[Gershgorin Circle Theorem]\cite{Gershgorin}\label{thm:gershgorin}
Let $B=(b_{ij})$ be a complex $n \times n$ matrix. For $i\in \{1, \dots , n\}$ let $\rho_i = \sum_{j\neq i} |b_{ij}|$ be the sum of the absolute values of the non-diagonal entries in the $i$-th row, and let $D(b_{ii}, \rho_i) \subset \mathbb C$ be the closed disk centered at $b_{ii}$ with radius $\rho_i$. Every eigenvalue of $B$ lies in at least one of the the discs $D(b_{ii}, \rho_i)$. 
\end{theorem}

For the rest of the article we use the notation $e_k := \begin{pmatrix} 0 & \cdots & 0& 1 & 0 & \cdots & 0
\end{pmatrix}^T$.

\begin{lemma}\label{unique}
For each $n\in \mathbb N$ there exists a unique positive eigenvalue of $\tilde{A}_{2n +1}$. Moreover, the corresponding (real) eigenspace $E^+ \subset \R^{2n+2}$ is $1$-dimensional and its submodule of integral elements $\Z^{2n+2} \cap E^+$ is generated by $e_1 + e_{2n+2}$.  
\end{lemma}

\begin{proof}
Let $w := e_1 + e_{2n+2}$ and notice that $\tilde{A}_{2n +1} w = (2n -1)w$, so that $\lambda := 2n -1$ is a positive eigenvalue of $\tilde{A}_{2n+1}$. We will argue that all other eigenvalues are non-positive. As a first step, we show that the eigenspace associated to $\lambda = 2n-1$ has dimension $1$. Note that 
\[
\tilde{A}_{2n+1} - \lambda I = \begin{pmatrix}
-2n - 1 & 1 & 1 & \cdots & 1 & 2n+1 \\
1 & -2n - 1 & 0 & \cdots & 0 & -1 \\
1 & 0 & -2n - 1 & \cdots & 0 & -1 \\
\vdots & \vdots & \vdots & \ddots & \vdots & \vdots\\
1 & 0 & 0 & \cdots & -2n - 1 & -1 \\
2n+1 & -1 & -1 & \cdots & -1 & -2n - 1
\end{pmatrix}.
\]
Replacing the last row with the sum of the first and last rows produces a row-equivalent matrix with a final row of $0$ entries. Therefore, to verify that $\tilde{A}_{2n+1} - \lambda I$ has a $1$-dimensional kernel, it suffices to establish invertibility of the upper-left $(2n+1)\times (2n+1)$ block 
\[
\begin{pmatrix}
-2n - 1 & 1 & 1 & \cdots & 1  \\
1 & -2n - 1 & 0 & \cdots & 0  \\
1 & 0 & -2n - 1 & \cdots & 0 \\
\vdots & \vdots & \vdots & \ddots & \vdots \\
1 & 0 & 0 & \cdots & -2n - 1  
\end{pmatrix}.
\]
Scaling the first row by $2n+1$ and then adding rows $2$ through $2n+1$ to the result gives the lower-triangular matrix 
\[
\begin{pmatrix}
-(2n+1)^2 + 2n & 0 & 0 & \cdots & 0  \\
1 & -2n - 1 & 0 & \cdots & 0  \\
1 & 0 & -2n - 1 & \cdots & 0 \\
\vdots & \vdots & \vdots & \ddots & \vdots \\
1 & 0 & 0 & \cdots & -2n - 1  
\end{pmatrix}.
\]
Invertibility readily follows, as $-(2n+1)^2 + 2n = -4n^2 - 2n - 1 < 0$. This establishes that the eigenspace $E^+$ of $\tilde{A}_{2n+1}$ associated to $\lambda = 2n-1$ is $1$-dimensional as a real subspace of $\R^{2n+2}$. 

Next, we observe that $w = e_1 + e_{2n+2}$ not only spans $E^+$ as a real subspace but is also an integrally minimal element, consequently generating $\Z^{2n+2} \cap E^{+}$. 

Finally, we argue that there are no other positive eigenvalues. Since $\tilde{A}_{2n+1}$ is real and symmetric, all other eigenvectors lie in the orthogonal complement of $w$, given by 
\[
w^{\perp} = \mathrm{Span} \{f_1:= -e_1 + e_{2n+2}, f_2:= e_2, f_3:= e_3, \dots , f_{2n+1}:= e_{2n+1}\}
\]
Consider the $(2n+1) \times (2n+1)$ restriction matrix $B_{2n+1}:= \tilde{A}_{2n+1}\mid_{w^{\perp}}$, computed component-wise by $b_{ij}:= f_j^T \tilde{A}_{2n+1} f_i$:
\[
B_{2n+1} = \begin{pmatrix}
-(4n+6) & -2 & -2 & \dots & -2  \\
-2 & -2 & 0 & \cdots & 0  \\
-2 & 0 & -2 & \cdots & 0  \\
-2 & \vdots & \vdots & \ddots & \vdots  \\
-2 & 0 & 0 & \cdots & -2  \\
\end{pmatrix}
\]
Appealing to the Gershgorin Circle Theorem, the disk $D(b_{ii}, \rho_i) $ for $i\in \{2, \dots , 2n+1\}$ is  centered at $b_{ii}=-2\in \mathbb C$ with radius $\rho_i = 2$. The remaining disk is centered at $b_{11}= -(4n+6)\in \C$ with radius $\rho_1 = 2n\cdot 2= 4n$. Since none of these disks intersect $\mathbb R^+ = \{(x,y)\in \mathbb C : x>0\}$, \cref{thm:gershgorin} implies that $B_{2n+1}$ has no positive eigenvalues, completing the proof of the lemma. 
\end{proof}

\begin{proof}[Proof of (2) of \cref{thm:main}.]
Suppose for the sake of contradiction that there exists $m \neq n$ such that $\Sigma_{2m+1} \leq \Sigma_{2n+1}$. Since $H_1(\Sigma_{2n+1}; \Z) \cong \Z^{2n+2}$ and $H_1(\Sigma_{2m+1}; \Z)$ must inject into $H_1(\Sigma_{2n+1}; \Z)$, we may assume $m < n$. 

Consider the respective symmetrized Seifert forms $\tilde{V}_{2m+1}$ and $\tilde{V}_{2n+1}$ represented by the matrices $\tilde{A}_{2m+1}$ and $\tilde{A}_{2n+1}$. By \cref{unique-eigenvalue}, both $\tilde{A}_{2m+1}$ and $\tilde{A}_{2n+1}$ have unique positive eigenvalues with $1$-dimensional eigenspaces $E^+_{2m+1}$ and $E^+_{2n+1}$, respectively.  

For any $k$, let $T_{2k+1}\subset \Sigma_{2k+1}$ denote the $(2,2k+1)$ torus knot drawn in red in \cref{fig:torus_family} (also identified with $\Lambda_{2k+1}$ in \cref{sec:qp_fib}). Consider the induced quadratic form $Q_n: H_1(\Sigma_{2n+1}; \mathbb Z) \to \mathbb Z$ defined by $Q_n(a) := \tilde{V}_{2n+1}(a,a)$. Since $\Sigma_{2m+1}\leq \Sigma_{2n+1}$, it follows that $\Sigma_{2n+1}$ supports both $T_{2n+1}$ and $T_{2m+1}\subset \Sigma_{2m+1}\subset \Sigma_{2n+1}$ as simple closed curves.

By definition of $Q_n$, for any closed curve $\gamma\subset \Sigma_{2n+1}$, the integer $Q_n([\gamma])$ is twice the $\Sigma_{2n+1}$-surface framing of $\gamma$. By construction of the family $\Sigma_{2k+1}$, the $\Sigma_{2k+1}$-surface framing integer of $T_{2k+1}$ is the Thurston-Bennequin invariant of its max-tb representative $\Lambda_{2k+1}$. Since $\Sigma_{2m+1}\leq \Sigma_{2n+1}$, the $\Sigma_{2m+1}$-framing of $T_{2m+1}$ agrees with that of the $\Sigma_{2n+1}$-framing. It follows that 
\begin{align*}
    Q_n([T_{2m+1}]) &= 2(2m-1), \\
    Q_n([T_{2n+1}]) &= 2(2n-1).
\end{align*}
With respect to the basis of $H_1(\Sigma_{2n+1}; \Z)$ in \cref{fig:torus_family_curves}, we have $[T_{2n+1}] = [\gamma_1] + [\gamma_{2n+2}] = w$, where $w$ is the eigenvector with unique positive eigenvalue and eigenspace $E_{2n+1}^+$ from \cref{unique}. Since $Q_n([T_{2m+1}]) > 0$ and $Q_n$ has a $1$-dimensional positive definite invariant subspace, it must be the case that $[T_{2m+1}]\in E_{2n+1}^+$. Moreover, $[T_{2m+1}]\in H_1(\Sigma_{2n+1}; \Z)$ is an integral element of the real subspace $E_{2n+1}^+$. Since the integral elements of $E_{2n+1}^+$ are generated by $[T_{2n+1}]$ by \cref{unique}, we have $[T_{2m+1}] = c\cdot [T_{2n+1}]$ for some integer $c\in \Z$. 

Consequently, $Q_n([T_{2m+1}]) = c^2 \cdot Q_n([T_{2n+1}])$ and thus $2(2m-1)= c^2\cdot 2(2q-1).$ Some rearrangement gives $c^2 = \frac{2m-1}{2n-1} <1$. Since $c\in \Z$, this is a contradiction.  
\end{proof}

\bibliography{references}
\bibliographystyle{amsalpha}

\end{document}